\documentclass[a4paper,12pt,reqno
]{amsart}

\usepackage[T2A]{fontenc}
\usepackage[cp1251]{inputenc}
\usepackage[english]{babel}

\usepackage{amssymb,upref}
\usepackage{url}

\usepackage{orcidlink}
\let\orcidID\orcidlink
\hypersetup{hidelinks}

\allowdisplaybreaks

\theoremstyle{plain}
\newtheorem{theorem}{Theorem}
\newtheorem{lemma}{Lemma}
\theoremstyle{remark}
\newtheorem{remark}{Remark}

\newcommand{\N}{\mathbb{N}}

\begin{document}

\title[Tribin functions related to $s$-symbol encodings]%
{A class of Tribin functions\\
  related to $s$-symbol encodings of numbers\\
  with a zero redundancy}
\author[Mykola Pratsiovytyi et al.]{Mykola Pratsiovytyi%
  \,\orcidID{0000-0001-6130-9413}}
\address{Mykhailo Drahomanov Ukrainian State University,
  Institute of Mathematics of the National Academy of Sciences of Ukraine}
\email{pratsiovytyi@imath.kiev.ua}
\email{prats4444@gmail.com}
\author[]{Sofiia Ratushniak%
  \,\orcidID{0009-0005-2849-6233}}
\address{Institute of Mathematics of the National Academy of Sciences of Ukraine,
  Mykhailo Drahomanov Ukrainian State University}
\email{ratush404@gmail.com}
\author[]{Oleksandr Baranovskyi%
  \,\orcidID{0009-0001-4495-1778}}
\address{Institute of Mathematics of the National Academy of Sciences of Ukraine}
\email{baranovskyi@imath.kiev.ua}
\author[]{Iryna Lysenko%
  \,\orcidID{0009-0000-5299-7787}}
\address{Mykhailo Drahomanov Ukrainian State University}
\email{i.m.lysenko@udu.edu.ua}


\keywords{%
  Encoding of numbers with a finite alphabet,
  $g$-representation of numbers,
  $g$-cylinder,
  $s$-adic representation of numbers,
  topologically equivalent representations,
  continuous nowhere monotonic function,
  continuous non-differentiable function,
  level set of a function.\\
    This work was supported by a grant from the Simons
Foundation (SFI-PD-Ukraine-00014586, M.P., S.R., O.B.)%
  }

\begin{abstract}
  In this paper, we consider a continuum class
  of continuous nowhere monotonic functions
  that generalize certain non-differentiable functions,
  including the Bush function, Wunderlich function,
  continuous Cantor projectors, Tribin function, etc.

  We consider a construction of the function
  related to $s$-symbol representations of numbers with a zero redundancy
  that are topologically equivalent to the classical $s$-adic representation
  (a value of the function has a two-symbol representation).
  Moreover, the condition on the first digit of a representation
  for the value of the function is more general than conditions
  considered before.

  The main object of study is a continuous function defined by equality
  \begin{gather*}
    f(\Delta^{s^*}_{\alpha_1\alpha_2\ldots\alpha_n\ldots})
    = \Delta^{2^*}_{\beta_1\beta_2\ldots\beta_n\ldots},
    \quad
    \alpha_n \in \{ 0, 1, 2, \ldots, s - 1 \} \equiv A_s, \\
    \beta_1 = \begin{cases}
      0 & \text{if $\alpha_1 \in A_0$}, \\
      1 & \text{if $\alpha_1 \in A_1$},
    \end{cases}
    \quad
    \beta_{n+1} = \begin{cases}
      \beta_n     & \text{if $\alpha_{n+1} = \alpha_n$}, \\
      1 - \beta_n & \text{if $\alpha_{n+1} \neq \alpha_n$}.
    \end{cases}
  \end{gather*}
  where
  $\Delta^{s^*}_{\alpha_1\alpha_2\ldots\alpha_n\ldots}$
  is an $s$-symbol representation of a number $x \in [0, 1]$
  that is topologically equivalent to the classical $s$-adic representation,
  $\Delta^{2^*}_{\beta_1\beta_2\ldots\beta_n\ldots}$
  is a two-symbol representation
  that is topologically equivalent to the classical binary representation,
  and $A_0 \cup A_1 = A_s$, $A_0 \neq A_s \neq A_1$.
  We prove that the function $f$ has finite and continuum level sets,
  its range is the interval $[0, 1]$,
  the image of a cylinder under the mapping $f$ is a cylinder,
  and the function has an unbounded variation.
\end{abstract}

\maketitle

\section*{Introduction}

In this paper, we discuss real-valued functions of a real argument
defined on an interval.
If a function does not have any (arbitrarily small) interval of monotonocity,
it is called \emph{nowhere monotonic}.
This property of a function is necessary
but not sufficient for it to be nowhere differentiable.
Among nowhere monotonic functions,
there exist functions of a bounded and unbounded variation~\cite{%
  Pratsiovytyi:2022:OCS%
},
non-differentiable functions
and functions with ``satisfactory'' differential properties,
including singular functions
(functions whose derivative is equal to zero almost everywhere
with respect to the Lebesgue measure)~\cite{%
  Pratsiovytyi:2022:OCS,%
  Pratsiovytyi:2011:NMS:en,%
  Turbin:1992:FMF:en%
}.

Now we can observe high-level interest
to continuous nowhere monotonic and non-differentiable functions~\cite{%
  Chen:2020:FTS,%
  Jarnicki:2015:CND,%
  Massopust:1997:FFA%
}
both in function theory,
and in fractal analysis~\cite{%
  Pratsiovytyi:2002:FVO:en,%
  Pratsiovytyi:2009:DFV:en,%
  Pratsiovytyi:2013:FPF%
},
theory of locally complicated probability measures~\cite{%
  Pratsiovytyi:1998:FPD:en%
}
as well as theory of dynamical systems.
Study of nowhere monotonic functions are essentially expanded,
and new techniques, methods, and tools for their description and study appear~\cite{%
  Ratushniak:2023:NNMA2:en,%
  Ratushniak:2023:NNMA:en%
}.
One of the sufficiently productive tools for defining functions
is various systems of encoding for real numbers (numeral systems)~\cite{%
  Galambos:1976:RRN,%
  Schweiger:1995:ETF,%
  Pratsiovytyi:2022:DSK:en%
}
and ``transducers'' of digits of one representation to another
(automata with finite memory)~\cite{%
  Pratsiovytyi:2013:SSN:en,%
  Pratsiovytyi:2022:CA2S%
}.
Such a way of defining a function goes back to papers by W.~Sierpi\'nski~\cite{%
  Sierpinski:1916:CCC%
},
T.~Takagi~\cite{%
  Takagi:1901:SEC%
}
and others.

One of the fruitful ideas for defining nowhere differentiable
and thus nowhere monotonic functions
is a technique initially used in papers~\cite{%
  Wunderlich:1952:USN,%
  Bush:1952:CFD,%
  Pratsiovytyi:1989:NKP:en,%
  Turbin:1992:FMF:en%
}
and is wide-spread in various papers
using various systems of representation for numbers~\cite{%
  Pratsiovytyi:2019:NMF:en,%
  Pratsiovytyi:2021:GTF,%
  Ratushniak:2023:NNMA2:en,%
  Ratushniak:2023:NNMA:en%
}.
In this paper, we generalize certain known results~\cite{%
  Bush:1952:CFD,%
  Bush:1962:LRF,%
  Bush:1971:PFC,%
  Wunderlich:1952:USN,%
  Pratsiovytyi:1989:NKP:en,%
  Pratsiovytyi:2002:FVO:en,%
  Panasenko:2009:OKN:en,%
  Panasenko:2009:HBD,%
  Pratsiovytyi:2021:GTF%
}
considering construction of a function
defined by abstract $s$-symbol representations
that are topologically equivalent to the classical $s$-adic representation
for the argument
and two-symbol representations for the value of the function.

\section{An encoding of real numbers and their $g$-representation}

Suppose $1 < s$ is a fixed natural number,
$A_s = \{ 0, 1, 2, \ldots, s-1 \}$ is an $s$-adic alphabet (a set of digits),
and $L_s = A_s \times A_s \times \ldots$ is a space of sequences
of elements of the alphabet.

Let us recall~\cite{%
  Pratsiovytyi:2022:DSK:en%
}
that an \emph{encoding of numbers} of the closed interval $[0, 1]$
with an alphabet $A_s$
is a surjective mapping $g$ of the space $L_s$ into the closed interval $[0, 1]$:
$L_s \xrightarrow{g} [0, 1]$.

If $g((\alpha_n)) = x \in [0, 1]$,
then the sequence $(\alpha_n)$ is called a \emph{$g$-code} of a number $x$.
This is written as
$x = \Delta^g_{\alpha_1\alpha_2\ldots\alpha_n\ldots}$
and called a \emph{$g$-representation}.
Moreover, $\alpha_n$ is called an $n$th digit of this representation.

In the topological and metric theory of a representation of numbers,
the important notion is a notion of a $g$-cylinder.
A $g$-cylinder of rank $m$ with base $c_1c_2\ldots c_m$ is a set
\[
  \Delta^g_{c_1\ldots c_m}
  = \{ x \colon x = \Delta^g_{c_1\ldots c_m\alpha_1\ldots\alpha_n\ldots}, \;
  (\alpha_n) \in L_s \}.
\]
It is evident that
\begin{gather*}
  \Delta \equiv [0, 1]
  = \bigcup_{i=0}^{s-1} \Delta^g_{i}
  = \bigcup_{i_1=0}^{s-1} \bigcup_{i_2=0}^{s-1} \Delta^g_{i_1i_2}
  = \ldots
  = \bigcup_{i_1=0}^{s-1} \ldots \bigcup_{i_n=0}^{s-1} \Delta^g_{i_1i_2\ldots i_n}, \\
  \bigcap_{n=1}^\infty \Delta^g_{c_1c_2\ldots c_n}
  = \Delta^g_{c_1c_2\ldots c_n\ldots}
  = x.
\end{gather*}

\begin{remark}
  If
  $x = \Delta^g_{\alpha_1\alpha_2\ldots\alpha_n\ldots}$,
  then
  $x \in \Delta^g_{\alpha_1\alpha_2\ldots\alpha_n}$
  for all $n \in \N$ and
  \[
    x = \bigcap\limits_{n=1}^\infty \Delta^g_{\alpha_1\alpha_2\ldots\alpha_n},
  \]
  and thus,
  $|\Delta^g_{\alpha_1\alpha_2\ldots\alpha_n}| \to 0$
  as
  $n \to \infty$.
\end{remark}

The system of $g$-encoding is said to have a \emph{zero redundancy}
if almost all numbers have a unique $g$-representation
(perhaps, except for a countable set of numbers having two representations).
Points that have two $g$-representations are called \emph{$g$-binary},
and those that have a unique $g$-representation are called \emph{$g$-unary}.
Digits $\alpha_n(x)$ of the representation of a $g$-unary point
are well-defined functions of a number $x$.

An encoding ($g$-representation) of numbers is called \emph{continuous}
if every $g$-cylinder is an interval,
and cylinders of the same rank do not have common internal points
(do not overlap).
In the sequel, we only discuss continuous $g$-representations
with a zero redundancy.
The simplest example of a continuous encoding for numbers
of the closed interval $[0, 1]$ that has a zero redundancy
is a classical $s$-adic representation:
\[
  x = \frac{\alpha_1}{s} + \frac{\alpha_2}{s^2} + \ldots
    + \frac{\alpha_n}{s^n} + \ldots
  \equiv \Delta^s_{\alpha_1\alpha_2\ldots\alpha_n\ldots}.
\]
Having one of the $s$-symbol encodings (representations) for numbers
of the unit interval,
we can obtain a new encoding
using a continuous strictly increasing probability distribution function $\gamma$:
\[
  \Delta^r_{c_1\ldots c_m} = \gamma(\Delta^g_{c_1\ldots c_m})
  \quad
  \text{for all $\Delta^g_{c_1\ldots c_m}$}.
\]
Thus it is clear that there exist infinitely many $s$-symbol encodings.

\begin{lemma}
  \label{lem:g.limit.equiv}
  Let $x_0$ be a $g$-unary point.
  Then the condition $x \to x_0$ is equivalent to the condition $m \to \infty$
  where $\alpha_m(x) \neq \alpha_m(x_0)$
  but $\alpha_i(x) = \alpha_i(x_0)$ for $i < m$.
\end{lemma}

\begin{proof}
  Let
  $x_0 = \Delta^g_{\alpha_1(x_0)\ldots\alpha_n(x_0)\ldots}$ and $x \neq x_0$.
  If
  $x_0 \neq x = \Delta^g_{\alpha_1(x)\ldots\alpha_n(x)\ldots}$,
  then it is evident that there exists $m$ such that
  $\alpha_m(x) \neq \alpha_m(x_0)$
  but $\alpha_i(x) = \alpha_i(x_0)$ for $i = \overline{1, m-1}$.

  Since
  $x_0, x \in \Delta^g_{\alpha_1(x_0)\ldots\alpha_{m-1}(x_0)}$,
  we have
  \[
    |x - x_0| \leq |\Delta^g_{\alpha_1(x_0)\ldots\alpha_{m-1}(x_0)}| \to 0
    \quad
    \text{as $m \to \infty$}.
  \]
  Hence, $x \to x_0$.

  Now we prove the converse statement.
  Let $(x_n)$ be an arbitrary sequence
  that tends to $x_0$ as $n \to \infty$.
  Assume the converse, namely, $m$ does not tends to infinity
  that is there exists a natural $p$ such that $m < p$
  for any $n \in \N$ and $m(n) \in \N$.
  Since
  \[
    x_0 = \Delta^g_{\alpha_1(x_0)\ldots\alpha_n(x_0)\ldots}
    = \bigcap_{n=1}^\infty \Delta^g_{\alpha_1(x_0)\ldots\alpha_n(x_0)}
      \subset \Delta^g_{\alpha_1(x_0)\ldots\alpha_p(x_0)},
  \]
  we have
  $x_0 \in \Delta^g_{\alpha_1(x_0)\ldots\alpha_p(x_0)}$,
  but
  $x_n \notin \Delta^g_{\alpha_1(x_0)\ldots\alpha_p(x_0)}$.
  Then
  \[
    |x_n-x_0| \geq \min\{
      x_0 - \min\{ \Delta^g_{\alpha_1\ldots\alpha_p} \}%
      ,
      \max\{ \Delta^g_{\alpha_1\ldots\alpha_p} \} - x_0
    \}
    = \text{const}.
  \]
  This contradicts the condition $x_n \to x_0$ as $n \to \infty$.

  The obtained contradiction proves
  that the condition $x_n \to x_0$ implies $m \to \infty$.
\end{proof}

Two encodings with a zero redundancy
($g_1$-representation and $g_2$-representation of numbers)
using a common alphabet $A$
are called \emph{topologically equivalent}
if the function (projector)
$\Delta^{g_1}_{\alpha_1\alpha_2\ldots\alpha_n\ldots}
\to \Delta^{g_2}_{\alpha_1\alpha_2\ldots\alpha_n\ldots}$
is continuous and strictly monotonic.
The classical binary representation and $Q_2$-representation
as well as nega-binary representation and $A_2$-continued representation
are topologically equivalent.
However, the nega-binary representation is not topologically equivalent
to classical binary representation.

\section{The main object of study}

Suppose $2 < s$ is a fixed natural number,
$\Delta^{s^*}_{\alpha_1\alpha_2\ldots\alpha_n\ldots}$
is an $s$-symbol representation of a number $x \in [0, 1]$
that is topologically equivalent to the classical $s$-adic representation:
\[
  x = \Delta^s_{\alpha_1\alpha_2\ldots\alpha_n\ldots}
  = \frac{\alpha_1}{s} + \frac{\alpha_2}{s^2} + \ldots
  + \frac{\alpha_n}{s^n} + \ldots,
  \quad
  (\alpha_n) \in L_s,
\]
and
$\Delta^{2^*}_{\beta_1\beta_2\ldots\beta_n\ldots}$
is a two-symbol representation of a number $x \in [0, 1]$
that is topologically equivalent to the classical binary representation,
$(\beta_n) \in L_2$.

Let $A_0 \cup A_1 = A_s$, $A_0 \neq A_s \neq A_1$.
Let us define a function $f$ by the following equality:
\begin{equation}
  \label{eq:g.tribin.def.f}
  f(\Delta^{s^*}_{\alpha_1\alpha_2\ldots\alpha_n\ldots})
  = \Delta^{2^*}_{\beta_1\beta_2\ldots\beta_n\ldots},
\end{equation}
where
\begin{align}
  \label{eq:g.tribin.def.b1}
  \beta_1 &= \begin{cases}
    0 & \text{if $\alpha_1 \in A_0$}, \\
    1 & \text{if $\alpha_1 \in A_1$},
  \end{cases} \\
  \label{eq:g.tribin.def.bn}
  \beta_{n+1} &= \begin{cases}
    \beta_n     & \text{if $\alpha_{n+1} = \alpha_n$}, \\
    1 - \beta_n & \text{if $\alpha_{n+1} \neq \alpha_n$}.
  \end{cases}
\end{align}

From the definition of digits $(\beta_n)$ of the $2^*$-representation
for the value of the function we see
that the digit $\beta_n$ depends on $n$ first digits of the argument,
i.e., is a function of $n$ variables, $\varphi_n \colon A_s^n \to A_s$.
Moreover,
$\beta_1 = \varphi_1(\alpha_1)$,
$\beta_{n+1} = \varphi_{n+1}(\alpha_1,\ldots,\alpha_n,\alpha_{n+1})
= \delta(\beta_n,\alpha_n,\alpha_{n+1})$.

\section{Well-definedness of the function and its continuity}

\begin{remark}
  If calculations of the values of a function
  by formulae~\eqref{eq:g.tribin.def.f}--\eqref{eq:g.tribin.def.bn}
  of two different representations of an $s^*$-binary point
  \[
    \Delta^{s^*}_{c_1\ldots c_m(0)}
    = \Delta^{s^*}_{c_1\ldots c_{m-1}[c_m-1](s-1)}
  \]
  do not coincide,
  we will say that the function $f$ is not well defined at this point.
  We can easily overcome this disadvantage using an agreement
  to use only one of two representations,
  but it does not ensure a continuity of the function at this point.
\end{remark}

\begin{lemma}
  The function $f$ is well defined
  by equalities~\eqref{eq:g.tribin.def.f}--\eqref{eq:g.tribin.def.bn}.
\end{lemma}

\begin{proof}
  First we consider binary points of the 1st rank, i.e., points of the form
  $\Delta^{s^*}_{a(0)} = \Delta^{s^*}_{[a-1](s-1)}$.
  If $a \in A_i$, $i \in \{ 0, 1 \}$ then
  \[
    f(\Delta^{s^*}_{a(0)}) = \Delta^{2^*}_{i(1-i)},
    \quad
    f(\Delta^{s^*}_{[a-1](s-1)}) = \begin{cases}
      \Delta^{2^*}_{0(1)} & \text{if $a-1 \in A_0$}, \\
      \Delta^{2^*}_{1(0)} & \text{if $a-1\in A_1$}.
    \end{cases}
  \]

  Now we consider binary points of the $m$th rank:
  \[
    \Delta^{s^*}_{a_1\ldots a_{m-1}a_m(0)}
    = \Delta^{s^*}_{a_1\ldots a_{m-1}[a_m-1](s-1)}.
  \]
  Let
  $\varphi_i(\alpha_1,\ldots,\alpha_i) = \beta_i$
  for
  $i = \overline{1,m}$.
  Then
  \begin{align*}
    f(\Delta^{s^*}_{a_1\ldots a_{m-1}a_m(0)})
    &= \Delta^{2^*}_{\beta_1\ldots\beta_{m-1}\beta_m(1-\beta_m)}, \\
    f(\Delta^{s^*}_{a_1\ldots a_{m-1}[a_m-1](s-1)})
    &= \Delta^{2^*}_{\beta_1\ldots\beta_{m-1}[1-\beta_m](\beta_m)}.
  \end{align*}

  It is evident that formulae~\eqref{eq:g.tribin.def.f}--\eqref{eq:g.tribin.def.bn}
  give equal values for both representations of $s^*$-binary points,
  in both cases (either $\beta_m = 0$ or $\beta_m = 1$).
  Thus the function $f$ is well defined.
  This proves the lemma.
\end{proof}

\begin{remark}
  Well-definedness of the function can be failed
  if a binary points could be found such that values of the function
  calculated by formulae~\eqref{eq:g.tribin.def.f}--\eqref{eq:g.tribin.def.bn}
  for different representations
  turn out to be different.
\end{remark}

\begin{theorem}
  \label{thm:g1g2.func.continuous}
  Suppose given $g_1$-representation and $g_2$-representation for numbers
  of the interval $[0, 1]$
  are continuous representations with a zero redundancy.
  A necessary and sufficient condition for the function $\psi$
  defined by equality
  \begin{equation}
    \label{eq:g1g2.func.def}
    \psi(x = \Delta^{g_1}_{\alpha_1\alpha_2\ldots\alpha_n\ldots})
    = \Delta^{g_2}_{\beta_1\beta_2\ldots\beta_n\ldots},
  \end{equation}
  where
  $\beta_n = \varphi_n(\alpha_1,\ldots,\alpha_n)$,
  to be continuous at every point of the invterval $[0, 1]$
  is that the function be well-defined at every $g_1$-binary point.
\end{theorem}

\begin{proof}
  1.~Necessity.
  Suppose the function $\psi$ is defined and continuous
  at every point of the interval $[0, 1]$,
  i.e., for all $x_0 \in [0, 1]$
  \[
    \lim_{x \to x_0} \psi(x) = \psi(x_0),
    \quad
    \text{or equivalently}
    \quad
    \lim_{x \to x_0} |\psi(x) - \psi(x_0)| = 0.
  \]

  Consider any binary point
  $x_0 = \Delta^{g_1}_{c_1\ldots c_ma_1\ldots a_n\ldots}
  = \Delta^{g_1}_{c_1\ldots c_mb_1\ldots b_n\ldots}$.
  Since the function $\psi$ is continuous at the point $x_0$,
  it must be defined at this point,
  and thus, formula~\eqref{eq:g1g2.func.def} gives equal results
  for different representations of the point $x_0$.

  2.~Sufficiency.
  Suppose
  $x_0 = \Delta^{g_1}_{c_1\ldots c_{m-1}a_1\ldots a_n\ldots}
  = \Delta^{g_1}_{c_1\ldots c_{m-1}b_1\ldots b_n\ldots}$
  is an arbitrary $g_1$-binary point ($a_1 \neq b_1$)
  and formula~\eqref{eq:g1g2.func.def} gives equal values
  for different $g_1$-representations of this point.
  We will prove that the function $\psi$ is continuous at the point $x_0$.

  Since $g_1$-encoding is continuous,
  either all $g_1$-cylinders
  \[
    \Delta^{g_1}_{c_1\ldots c_{m-1}a_1},
    \quad
    \Delta^{g_1}_{c_1\ldots c_{m-1}a_1a_2},
    \quad
    \ldots,
    \quad
    \Delta^{g_1}_{c_1\ldots c_{m-1}a_1\ldots a_n},
    \quad
    \ldots
  \]
  lie to the left of the point $x_0$,
  i.e.,
  $x_0 = \max\Delta^{g_1}_{c_1\ldots c_{m-1}a_1\ldots a_n}$
  for all $n \in \N$,
  or all mentioned $g$-cylinders lie to the right of the point $x_0$,
  i.e.,
  $x_0 = \min\Delta^{g_1}_{c_1\ldots c_{m-1}a_1\ldots a_n}$
  for all $n \in \N$.
  We have an analogous situation for another representation
  $\Delta^{g_1}_{c_1\ldots c_{m-1}b_1\ldots b_n\ldots}$
  of the point $x_0$.
  Without loss of generality we assume that
  \[
    x_0 = \max\Delta^{g_1}_{c_1\ldots c_{m-1}a_1\ldots a_n}
    = \min\Delta^{g_1}_{c_1\ldots c_{m-1}b_1\ldots b_n}.
  \]

  Let
  \[
    \psi(\Delta^{g_1}_{c_1\ldots c_{m-1}a_1\ldots a_n\ldots})
    = \Delta^{g_2}_{\beta_1\ldots\beta_{m-1}\beta_m\ldots\beta_{m+n-1}\ldots}.
  \]
  From an arbitrary sequence that tends to $x_0$ from the left
  we can select a strictly increasing subsequence $(x_n)$ such that
  $x_n \in \Delta^{g_1}_{c_1\ldots c_{m-1}a_1\ldots a_n}$
  but
  $x_n \notin \Delta^{g_1}_{c_1\ldots c_{m-1}a_1\ldots a_{n-1}}$.
  So
  \[
    |\psi(x_n) - \psi(\Delta^{g_1}_{c_1\ldots c_{m-1}a_1\ldots a_n\ldots})|
    \leq |\Delta^{g_2}_{\beta_1\ldots\beta_{m+n-1}}| \to 0,
    \quad
    \text{as $n \to \infty$}.
  \]
  Thus the function $\psi$ is left-continuous at the point $x_0$.

  We can prove that the function $\psi$ is right-continuous at the point $x_0$
  analogously
  but using another representation.
  Since the function $\psi$ is left-continuous and right-continuous
  at the point $x_0$, we have
  $\lim\limits_{x \to x_0} \psi(x) = \psi(x_0)$,
  i.e., it is continuous at every $g_1$-binary point.

  Now we will show
  that if the function is well defined at every $g_1$-binary point,
  then it is continuous at every $g_1$-unary point,
  including the points $0$ and $1$.
  Let
  $x_0 = \Delta^{g_1}_{c_1\ldots c_n\ldots}$
  be an arbitrary $g_1$-unary point,
  $\psi(x_0) = \Delta^{g_2}_{d_1\ldots d_n\ldots}$,
  $x_0 \neq x = \Delta^{g_1}_{\alpha_1(x)\ldots\alpha_n(x)\ldots}$.
  By Lemma~\ref{lem:g.limit.equiv},
  the condition $x \to x_0$ is equivalent to the condition $m \to \infty$,
  where $\alpha_m(x) \neq c_m$ but $\alpha_i(x) = c_i$ for $i < m$.

  Since the function is defined at the point $x_0$,
  continuity of the function $\psi$ at the point $x_0$ is equivalent to
  \[
    \lim_{m \to \infty} |\psi(x) - \psi(x_0)| = 0.
  \]
  However
  $|\psi(x) - \psi(x_0)| \leq |\Delta^{g_2}_{d_1\ldots d_{m-1}}| \to 0$
  as $m \to \infty$.
  Thus
  $\lim\limits_{x \to x_0} \psi(x) = \psi(x_0)$.
  This proves the theorem.
\end{proof}

\section{Range and level sets of the function}

Taking into account definition of the function
by equalities~\eqref{eq:g.tribin.def.f}--\eqref{eq:g.tribin.def.bn},
we can easily define the level set
$f^{-1}(\Delta^{2^*}_{\beta_1\beta_2\ldots\beta_n\ldots})$
of the function $f$ for the point
$\Delta^{2^*}_{\beta_1\beta_2\ldots\beta_n\ldots}$.
The elements of this set are points
$\Delta^{s^*}_{\alpha_1\alpha_2\ldots\alpha_n\ldots}$
whose representations satisfy the following conditions:
\begin{enumerate}
  \item
    if $\beta_1 = 0$, then $\alpha_1$ is an arbitrary digit of $A_0$,
  \item
    if $\beta_1 = 1$, then $\alpha_1$ is an arbitrary digit of $A_1$,
  \item
    if $\beta_2 = \beta_1$, then $\alpha_2 = \alpha_1$,
  \item
    if $\beta_2 \neq \beta_1$,
    then $\alpha_2$ is arbitrary but different from $\alpha_1$,
  \item
    $\alpha_{n+1} = \begin{cases}
      \alpha_n & \text{if $\beta_{n+1} = \beta_n$}, \\
      \text{arbitrary $a \in A_{1-\beta_n}$}
               & \text{if $\beta_{n+1} \neq \beta_n$}.
    \end{cases}$
\end{enumerate}

From the above-mentioned facts it follows
that the image of a cylinder under the mapping $f$ is a cylinder
and ``the most massive'' level sets are
$f^{-1}(\Delta^{2^*}_{(01)})$
and
$f^{-1}(\Delta^{2^*}_{(10)})$.
These sets have the cardinality of the continuum
and they are structurally fractal.
The following theorem partially summarizes the facts we said above.

\begin{theorem}
  The range $D_f$ of the function $f$ is the interval $[0, 1]$.
  The image of a cylinder is a cylinder.
  The function $f$ has finite and continuum levels.
\end{theorem}

\begin{proof}
  From the definition of the function it follows directly
  that its zeros are points $\Delta^{s^*}_{(a)}$, where $a \in A_0$.
  The level set for the point
  $y = 1 = \Delta^{2^*}_{(1)}$
  is the set
  $\{ \Delta^{s^*}_{(a)}, a \in A_1 \}$.
  Thus $D_f = [0, 1]$.
  The level sets for the points
  $0 = \Delta^{2^*}_{(0)}$ and $1 = \Delta^{2^*}_{(1)}$
  are finite.
  It is evident that the level set
  \[
    f^{-1}(\Delta^{2^*}_{(01)})
    = \{ \Delta^{s^*}_{\alpha_1\alpha_2\ldots\alpha_n\ldots}, \;
      \alpha_1 \in A_0, \; \alpha_{n+1} \neq \alpha_n \}
  \]
  for the point $\Delta^{2^*}_{(01)}$ is continuum and structurally fractal
  because
  \[
    f^{-1}(\Delta^{2^*}_{(01)})
    = \bigcup_{\alpha_1 \in A_0} \bigcup_{\alpha_2 \neq \alpha_1}
    \Delta^{s^*}_{\alpha_1\alpha_2} \cap f^{-1}(\Delta^{2^*}_{(01)}),
  \]
  and the sets consisting in the union are ``topologically self-similar.''
\end{proof}

\section{Nowhere monotonocity of the function}

\begin{theorem}
  The function $f$
  defined by equalities~\eqref{eq:g.tribin.def.f}--\eqref{eq:g.tribin.def.bn}
  is nowhere monotonic.
\end{theorem}

\begin{proof}
  We will use the definition of a nowhere monotonic function,
  namely, we will prove
  that the function $f$ is not monotonic on every interval from its domain.

  Let $(a, b)$ be an arbitrary interval that is a subset of $[0, 1]$.
  It is evident that there exists $\Delta^{s^*}$-cylinder
  that is a subset of $(a, b)$.
  Suppose this is the cylinder $\Delta^{s^*}_{c_1\ldots c_m}$.
  Let us consider two cases.

  1. If $0 \neq c_m \neq s - 1$, we consider three points:
  \[
    x_0 = \Delta^{s^*}_{c_1\ldots c_m(0)},
    \quad
    x_1 = \Delta^{s^*}_{c_1\ldots c_{m-1}(c_m)},
    \quad
    x_2 = \Delta^{s^*}_{c_1\ldots c_{m}(s-1)}.
  \]
  It is evident that $x_0 < x_1 < x_2$.
  In addition to that
  \[
    f(x_0) = \Delta^{s^*}_{\beta_1\ldots\beta_m(1-\beta_m)},
    \quad
    f(x_1) = \Delta^{s^*}_{\beta_1\ldots\beta_{m-1}(\beta_m)},
    \quad
    f(x_2) = \Delta^{s^*}_{\beta_1\ldots\beta_{m}(1-\beta_m)}.
  \]
  Hence we have $f(x_0) > f(x_1) < f(x_2)$ if $\beta_m = 0$
  and $f(x_0) < f(x_1)$ if $\beta_m = 1$.

  Thus the function $f$ is not monotonic on the cylinder
  $\Delta^{s^*}_{c_1\ldots c_m}$,
  so it is not monotonic on the interval $(a, b)$ as well.

  2. If $c_m \in \{ 0, s - 1 \}$, we consider the cylinder
  $\Delta^{s^*}_{c_1\ldots c_mc} \subset \Delta^{s^*}_{c_1\ldots c_m}
    \subset (a, b)$,
  where $0 < c < s - 1$.
  As we have proved the function is not monotonic on the cylinder
  $\Delta^{s^*}_{c_1\ldots c_mc}$,
  so it is not monotonic on $(a, b)$ as well.
  This proves the theorem.
\end{proof}

\section{Variation of the function}

\begin{lemma}
  \label{lem:g.tribin.cyl.ineq}
  For any cylinder $\Delta^{s^*}_{c_1c_2\ldots c_m}$ of the $m$th rank,
  there exists natural number $k$ such that
  \begin{equation}
    \label{eq:g.tribin.cyl.ineq}
    \sum_{\alpha_1=0}^{s-1} \sum_{\alpha_2=0}^{s-1} \ldots
      \sum_{\alpha_k=0}^{s-1} |f(\Delta^{s^*}_{c_1c_2\ldots c_m\alpha_1\ldots\alpha_k})|
    \geq 2 |f(\Delta^{s^*}_{c_1c_2\ldots c_m})|.
  \end{equation}
\end{lemma}

\begin{proof}
  Let
  $\Delta^{2^*}_{\beta_1\ldots\beta_m}=f(\Delta^{s^*}_{c_1c_2\ldots c_m})$.
  This cylinder contains images of all cylinders
  \begin{equation}
    \label{eq:g.tribin.cylinders}
    \Delta^{s^*}_{c_1\ldots c_m\alpha_1\ldots\alpha_k},
    \quad
    \alpha_i \in A_s,
    \quad
    i = \overline{1, k}.
  \end{equation}
  Among them, only the cylinder
  $f(\Delta^{s^*}_{c_1\ldots c_mc_m\ldots c_m})$
  has a single preimage.
  At the same time the cylinder
  $f(\Delta^{s^*}_{c_1\ldots c_mb_1\ldots b_k})$,
  where $c_m \neq b_1$, $b_i \neq b_{i+1}$, $i = \overline{1, k-1}$,
  has $2^k$ preimages.
  It is evident that there exists $k \in \N$ such that
  inequality~\eqref{eq:g.tribin.cyl.ineq} holds because
  \begin{enumerate}
  \item
    $|\Delta^{2^*}_{\alpha_1\alpha_2\ldots\alpha_n}| \to 0$
    as $n \to \infty$,
  \item
    the interval of the length
    $|\Delta^{2^*}_{\beta_1\ldots\beta_m}|
      - |f(\Delta^{s^*}_{c_1\ldots c_m\underbrace{\scriptstyle c_m\ldots c_m}_{k}})|$
    may be covered by other cylinders from the set~\eqref{eq:g.tribin.cylinders}
    at least three times.
  \end{enumerate}
  This proves the lemma.
\end{proof}

\begin{theorem}
  The function $f$
  defined by equalities~\eqref{eq:g.tribin.def.f}--\eqref{eq:g.tribin.def.bn}
  is a continuous function of an unbounded variation.
\end{theorem}

\begin{proof}
  We have already proved that the function is continuous.
  Since the image of a cylinder under the mapping $f$ is a cylinder
  and change in function on a cylinder (difference of maximum and minimum values)
  is equal to the length of its image,
  it is evident that variation $V(f)$ of the function $f$ is greater than
  the total length $W_k$ of images of all $s^*$-cylinders of rank $k$
  for all $k \in \N$, i.e.,
  \[
    V(f) > W_k
    = \sum_{\alpha_1=0}^{s-1} \sum_{\alpha_2=0}^{s-1} \ldots
      \sum_{\alpha_k=0}^{s-1} |f(\Delta^{s^*}_{\alpha_1\alpha_2\ldots\alpha_k})|.
  \]

  By Lemma~\ref{lem:g.tribin.cyl.ineq}, there exists $k_1 \in \N$ such that
  \[
    V(f) > V_1 \equiv W_{k_1} \geq 2.
  \]

  By the same lemma, there exists $k_2 \in \N$ such that
  \[
    V(f) > W_{k_1+k_2} \geq 2 V_1 > 2^2.
  \]
  Similar arguments on cylinders of rank $k_1 + k_2$ yield to
  \[
    V(f) > W_{k_1+k_2+k_3} \geq 2^3,
  \]
  and in general,
  \[
    V(f) > W_{k_1+k_2+\ldots+k_n} \geq 2^n
  \]
  for all $n \in \N$.
  Thus
  $V(f) = \lim\limits_{n \to \infty} 2^n = \infty$.
\end{proof}

\section{Concluding remarks}

If we change condition~\eqref{eq:g.tribin.def.bn} in the defintion of the function $f$
by the condition
\begin{equation}
  \label{eq:g.tribin.def.bn.flip}
  \beta_{n+1} = \begin{cases}
    \beta_n     & \text{if $\alpha_{n+1} \neq \alpha_n$}, \\
    1 - \beta_n & \text{if $\alpha_{n+1} = \alpha_n$},
  \end{cases}
\end{equation}
the function loses well-definedness at $s^*$-binary points.

Among the digits of the alphabet $A_s$ that are different from $0$,
there exists a digit $c$ such that
$c \in A_i$ and $c - 1 \in A_{1-i}$, $i \in \{ 0, 1 \}$.
Then values of the function $f$
calculated for different representations of the $s^*$-binary point
$\Delta^{s^*}_{c(0)} = \Delta^{s^*}_{[c-1](s+1)}$
are not equal.
Really,
\begin{gather*}
  f(\Delta^{s^*}_{c(0)})
  = \Delta^{2^*}_{\beta_1\beta_1([1-\beta_1]\beta_1)}, \\
  f(\Delta^{s^*}_{[c-1](s-1)})
  = \Delta^{2^*}_{\beta_1'\beta_1'([1-\beta'_1]\beta_1')}
  \neq f(\Delta^{s^*}_{c(0)}),
\end{gather*}
because $\beta_1 \neq \beta_1'$.
By Theorem~\ref{thm:g1g2.func.continuous},
it follows that the function loses its well-definedness and continuity.

Let us remark that we must give certain metric meaning
to $s$-symbol representations
in order to study differential properties of the function $f$.
However this is an idea for another paper.


\begin{thebibliography}{10}

\bibitem{Bush:1952:CFD}
K.~A. Bush, \emph{Continuous functions without derivatives}, Amer. Math.
  Monthly \textbf{59} (1952), no.~4, 222--225. \MR{0049278 (14,148b)}

\bibitem{Bush:1962:LRF}
K.~A. Bush, \emph{Locally recurrent functions}, Amer. Math. Monthly \textbf{69}
  (1962), no.~3, 199--206. \MR{0132131 (24 \#A1978)}

\bibitem{Bush:1971:PFC}
K.~A. Bush, \emph{Pathological functions, crinkly and wild}, J. Math. Anal.
  Appl. \textbf{35} (1971), no.~3, 559--562. \MR{0279247 (43 \#4970)}

\bibitem{Chen:2020:FTS}
Y.-G. Chen, \emph{Fractal texture and structure of central place systems},
  Fractals \textbf{28} (2020), no.~1, Paper No. 2050008.

\bibitem{Galambos:1976:RRN}
J.~Galambos, \emph{Representations of real numbers by infinite series}, Lecture
  Notes in Math., vol. 502, Springer-Verlag, Berlin, 1976. \MR{0568141 (58
  \#27873)}

\bibitem{Jarnicki:2015:CND}
M.~Jarnicki and P.~Pflug, \emph{Continuous nowhere differentiable functions.
  {The} monsters of analysis}, Springer Monogr. Math., Springer, Cham, 2015.
  \MR{3444902}

\bibitem{Massopust:1997:FFA}
P.~R. Massopust, \emph{Fractal functions and their applications}, Chaos
  Solitons Fractals \textbf{8} (1997), no.~2, 171--190. \MR{1431371
  (97m:28011)}

\bibitem{Panasenko:2009:HBD}
O.~B. Panasenko, \emph{{Hausdorff}--{Besicovitch} dimension of the graph of one
  continuous nowhere-differentiable function}, Ukrainian Math. J. \textbf{61}
  (2009), no.~9, 1448--1466. \MR{2752552 (2012f:28009)}

\bibitem{Panasenko:2009:OKN:en}
O.~B. Panasenko, \emph{A one-parameter class of continuous functions close to
  {Cantor} projectors}, Mat. Stud. \textbf{32} (2009), no.~1, 3--11 (in
  Ukrainian). \MR{2596792 (2011b:26009)}

\bibitem{Pratsiovytyi:1989:NKP:en}
M.~V. Pratsiovytyi, \emph{Continuous {Cantor} projectors}, Methods of the Study
  of Algebraic and Topological Structures, Kyiv State Pedagog. Inst., Kyiv,
  1989, pp.~95--105 (in Russian).

\bibitem{Pratsiovytyi:1998:FPD:en}
M.~V. Pratsiovytyi, \emph{Fractal approach to the study of singular probability
  distributions}, Natl. Pedagog. Mykhailo Dragomanov Univ. Publ., Kyiv, 1998
  (in Ukrainian).

\bibitem{Pratsiovytyi:2002:FVO:en}
M.~V. Pratsiovytyi, \emph{Fractal properties of one continuous nowhere
  differentiable function}, Nauk. Zap. Nats. Pedagog. Univ. Mykhaila
  Drahomanova. Fiz.-Mat. Nauky (2002), no.~3, 351--362 (in Ukrainian).

\bibitem{Pratsiovytyi:2011:NMS:en}
M.~V. Pratsiovytyi, \emph{Nowhere monotonic singular functions}, Nauk. Chasop.
  Nats. Pedagog. Univ. Mykhaila Drahomanova. Ser.~1. Fiz.-Mat. Nauky (2011),
  no.~12, 24--36 (in Ukrainian).

\bibitem{Pratsiovytyi:2022:DSK:en}
M.~V. Pratsiovytyi, \emph{Two-symbol systems of encoding of real numbers and
  their applications}, Nauk. Dumka, Kyiv, 2022 (in Ukrainian).

\bibitem{Pratsiovytyi:2021:GTF}
M.~V. Pratsiovytyi, O.~M. Baranovskyi, and {\relax Yu}.~P. Maslova,
  \emph{Generalization of the {Tribin} function}, J. Math. Sci. (N.~Y.)
  \textbf{253} (2021), no.~2, 276--288. \MR{4016749}

\bibitem{Pratsiovytyi:2019:NMF:en}
M.~V. Pratsiovytyi, N.~V. Cherchuk, {\relax Yu}.~{\relax Yu}. Vovk, and A.~V.
  Shevchenko, \emph{Nowhere monotonic functions related to representations of
  numbers by {Cantor} series}, Fractal Analysis and Related Problems: Trans.
  Inst. Math. NAS Ukraine \textbf{16} (2019), no.~3, 198--209 (in Ukrainian).

\bibitem{Pratsiovytyi:2022:CA2S}
M.~V. Pratsiovytyi, {\relax Ya}.~V. Goncharenko, I.~M. Lysenko, and S.~P.
  Ratushniak, \emph{Continued {$A_2$-fractions} and singular functions}, Mat.
  Stud. \textbf{58} (2022), no.~1, 3--12. \MR{4509563}

\bibitem{Pratsiovytyi:2022:OCS}
M.~V. Pratsiovytyi, {\relax Ya}.~V. Goncharenko, I.~M. Lysenko, and O.~V.
  Svynchuk, \emph{On one class of singular nowhere monotone functions}, J.
  Math. Sci. (N.~Y.) \textbf{263} (2022), no.~2, 268--281. \MR{4213150}

\bibitem{Pratsiovytyi:2013:SSN:en}
M.~V. Pratsiovytyi and A.~V. Kalashnikov, \emph{Self-affine singular and
  nowhere monotone functions related to the {$Q$-representation} of real
  numbers}, Ukrainian Math. J. \textbf{65} (2013), no.~3, 448--462 (in
  Ukrainian). \MR{3120031}

\bibitem{Pratsiovytyi:2009:DFV:en}
M.~Pratsiovytyi and O.~Panasenko, \emph{Differential and fractal properties of
  a class of self-affine functions}, Visn. Lviv. Univ. Ser. Mekh.-Mat. (2009),
  no.~70, 128--142 (in Ukrainian).

\bibitem{Pratsiovytyi:2013:FPF}
M.~Pratsiovytyi and N.~Vasylenko, \emph{Fractal properties of functions defined
  in terms of {$Q$-representation}}, Int. J. Math. Anal. (Ruse) \textbf{7}
  (2013), no.~64, 3155--3167. \MR{3162174}

\bibitem{Ratushniak:2023:NNMA2:en}
S.~P. Ratushniak, \emph{A continuous nowhere monotonic function defined in
  terms of the {$A_2$-continued} fraction representation of numbers}, Bukovyn.
  Mat. Zh. \textbf{11} (2023{\noopsort{a}}), no.~1, 126--133 (in Ukrainian).

\bibitem{Ratushniak:2023:NNMA:en}
S.~P. Ratushniak, \emph{A continuous nowhere monotonic function defined in
  terms of the {$A$-continued} fraction representation of numbers}, Bukovyn.
  Mat. Zh. \textbf{11} (2023{\noopsort{b}}), no.~2, 236--245 (in Ukrainian).

\bibitem{Schweiger:1995:ETF}
F.~Schweiger, \emph{Ergodic theory of fibred systems and metric number theory},
  Oxford Sci. Publ., Oxford Univ. Press, New York, 1995. \MR{1419320
  (97h:11083)}

\bibitem{Sierpinski:1916:CCC}
W.~Sierpi{\'n}ski, \emph{Sur une courbe cantorienne qui contient une image
  biunivoque et continue de toute courbe donn{\'e}e}, C. R. Acad. Sci. Paris
  \textbf{162} (1916), 629--632.

\bibitem{Takagi:1901:SEC}
T.~Takagi, \emph{A simple example of the continuous function without
  derivative}, T\=oky\=o S\=ugaku-Butsurigakkwai H\=okoku \textbf{1} (1901),
  176--177.

\bibitem{Turbin:1992:FMF:en}
A.~F. Turbin and M.~V. Pratsiovytyi, \emph{Fractal sets, functions, and
  probability distributions}, Nauk. Dumka, Kyiv, 1992 (in Russian). \MR{1353239
  (96f:28010)}

\bibitem{Wunderlich:1952:USN}
W.~Wunderlich, \emph{{Eine {\"u}berall stetige und nirgends differenzierbare
  Funktion}}, Elem. Math. \textbf{7} (1952), no.~4, 73--79. \MR{0049279
  (14,148c)}

\end{thebibliography}

\providecommand{\noopsort}[1]{}
\providecommand{\bysame}{\leavevmode\hbox to3em{\hrulefill}\thinspace}
\providecommand{\MR}{\relax\ifhmode\unskip\space\fi MR }
\providecommand{\MRhref}[2]{%
  \href{http://www.ams.org/mathscinet-getitem?mr=#1}{#2}
}
\providecommand{\href}[2]{#2}

\end{document}